\documentclass[11pt,reqno]{amsart}
\usepackage{amssymb,amsmath,amsthm, amsfonts}
\usepackage{graphicx}
\usepackage{epsfig}
\usepackage{tikz}

\def\disp{\displaystyle}

\def\dref#1{(\ref{#1})}

\theoremstyle{plain}
\newtheorem{theorem}{Theorem}[section]
\newtheorem{lemma}{Lemma}[section]

\theoremstyle{definition}

\newtheorem{remark}{Remark}[section]

\numberwithin{equation}{section}

\begin{document}

\author{Zhi-An Wang}
\address{Department of Applied Mathematics, Hong Kong Polytechnic University, Hung Hom, Hong Kong}
\email{mawza@polyu.edu.hk}

\author{Jiashan Zheng}
\address{School of Mathematics and Statistics Science, Ludong University, Yantai 264025,  P.R. China}
\email{zhengjiashan2008@163.com}

\title[parabolic Keller-Segel system with signal-dependent motilities]{Global boundedness of the fully parabolic Keller-Segel system with signal-dependent motilities}

\maketitle

%
%
\begin{quote}
{\bf Abstract}: This paper establishes the global uniform-in-time  boundedness of solutions to the following Keller-Setel system with signal-dependent diffusion and chemotaxis
\begin{eqnarray*}
 \left\{\begin{array}{ll}
  u_t=\nabla\cdot(\gamma(v)\nabla u-u\phi(v)\nabla v),\quad &
x\in \Omega, t>0,\\
 \disp{ v_t=d\Delta v- v+u},\quad &
x\in \Omega, t>0\\
 \end{array}\right.
\end{eqnarray*}
in a bounded domain
 $\Omega\subset\mathbb{R}^N(N\leq4)$ with smooth boundary, where the density-dependent motility functions $\gamma(v)$ and $\phi(v)$  denote the  diffusive and chemotactic coefficients, respectively.
 The model was originally proposed by Keller and Segel in \cite{Keller-1}
to describe the aggregation phase of Dictyostelium discoideum cells, where the two motility functions satisfy a proportional relation $\chi(v)=(\alpha -1)\gamma'(v)$ with $\alpha>0$ denoting the ratio of effective body length (i.e. distance between receptors) to the step size. The major technical difficulty in the analysis is the possible degeneracy of diffusion.
In this work, we show that if $\gamma(v)>0$ and $\phi(v)>0$ are smooth on $[0,\infty)$ and satisfy
$$\inf_{v\geq0}\frac{d\gamma(v)}{v\phi(v)(v\phi(v)+d-\gamma(v))_+}>\frac{N}{2},$$
then the above Keller-Segel system  subject to Neumann boundary conditions admits classical solutions uniformly bounded in time.  The main idea of proving our results is the estimates of a weighted functional $\int_\Omega u^{p}v^{-q}dx$ for $p>\frac{N}{2}$ by choosing a suitable exponent  $p$ depending on the unknown $v$, by which we are able to derive a uniform $L^\infty$-norm of $v$ and hence rule out the diffusion degeneracy.

\vspace{0.3cm}


\end{quote}
\section{Introduction}
In this paper, we consider the following Keller-Segel system with signal-dependent motilities
\begin{equation}\label{eq-1}
 \left\{\begin{array}{ll}
  u_t=\nabla\cdot(\gamma(v)\nabla u-u\phi(v)\nabla v),\quad &
x\in \Omega, t>0,\\
 \disp{ \tau v_t=d\Delta v +u- v},\quad &
x\in \Omega, t>0,\\
\partial_\nu u=\partial_\nu v=0,\quad &
x\in \partial\Omega, t>0,\\
\disp{u(x,0)=u_0(x)}, \ v(x,0)= v_0(x),\quad  &
x\in \Omega\\
 \end{array}\right.
\end{equation}
in a bounded domain $\Omega\subset R^N(N\geq1)$ with smooth boundary $\partial\Omega$, where $\tau=1$ and $d > 0$ is
the chemical diffusion rate and $\partial_\nu=\frac{\partial }{\partial \nu}$
denotes the derivative with respect to the unit outer normal vector $\nu$ of $\partial\Omega$;
$u(x, t)$ represents the cell density and $v(x, t)$ denotes the chemical signal concentration  at position $x$ and time $t$.
 The system \dref{eq-1} was proposed by Keller and Segel in \cite{Keller-1}
to describe the aggregation phase of Dictyostelium discoideum (Dd) cells in response to the chemical signal cyclic adenosine
monophosphate (cAMP) secreted by Dd cells, where the signa-dependent diffusivity $\gamma(v)$ and chemotactic sensitivity function $\phi(v)$ satisfy the following proportionality relation
\begin{equation}\label{relation}
\phi(v)=(\alpha-1)\gamma'(v),
\end{equation}
where $\alpha>0$ denotes the ratio of effective body length (i.e. distance between receptors) to the step size (see \cite{Keller-1} for details).

When $\gamma(v)$ and $\phi(v)$ are constant, the Keller-Segel system \dref{eq-1} is called the minimal chemotaxis model \cite{Nanjundiah} which has been extensively studied in the literature (see review paper \cite{Bellomo-1, Horstmann2710} and the references therein), and ``critical mass'' is perhaps the most prominent phenomenon amongst abundant results. When $\gamma(v) = 1$ and $\phi(v) =\frac{\chi}{v}$, the system \dref{eq-1} becomes the following so-called singular Keller-Segel system proposed in \cite{Keller-2}
\begin{equation}\label{eq-2}
 \left\{\begin{array}{ll}
  u_t=\Delta u-\chi\nabla\cdot(\frac{u}{v}\nabla v),\quad &
x\in \Omega, t>0,\\
 \disp{\tau v_t=\Delta v +u- v},\quad &
x\in \Omega, t>0,\\
\partial_\nu u=\partial_\nu v=0,\quad &
x\in \partial\Omega, t>0,\\
\disp{u(x,0)=u_0(x)}, \ v(x,0)= v_0(x),\quad  &
x\in \Omega
 \end{array}\right.
\end{equation}
which exhibits different behaviors from the minimal chemotaxis model and no `critical mass'' phenomenon was identified (cf. \cite{Bilerss-1,Fujie-1,Nagai-3, Lankeit-1,Winkler-5,Stinn-1}.
In particular, Nagai and   Senba (\cite{Nagai-3, Nagai-4}) proved that all {radial classical} solutions of  system \dref{eq-2} with $\tau=0$ are global-in-time if either $N\geq 3$ with $\chi <\frac{2}{N-2}$, or $N= 2$ with arbitrary $\chi > 0$, whereas the the radial solution may blow up if $\chi>\frac{2N}{N-2}$ and $N\geq 3$. When $\tau=1$, $N \geq 2$ and $\chi<\sqrt{\frac{2}{N}}$, Fujie (\cite{Fujie-5}) showed  that the system \dref{eq-2} possesses global classical solutions with uniform-in-time boundedness based on an observation that $v$ has a positive lower bounded. We refer to \cite{Fujie-2} for an exhaustive review of works related to the system \eqref{eq-2}. When $v$ is a nutrient consumed by bacteria, namely the second equation of \eqref{eq-2} is replaces by $v_t=\Delta v-uv^m (m>0$), the singular Keller-Segel was proposed in \cite{Keller-3} to describe the wave prorogation generated by bacterial chemotaxis (see \cite{Wang-review} for mathematical results).

In contrast to plenty of results developed for the minimal and singular Keller-Segel systems, when both $\gamma(v)$ and $\phi(v)$ are non-constant and ruled by the relation \eqref{relation}, the results obtained for \dref{eq-1} are much less. The mathematical analysis of \dref{eq-1} with signal-dependent diffusive and chemotactic coefficients  has to cope with considerable challenges caused by the possible degeneracy of diffusion (i.e.  $\gamma(v)$ may touch down to zero), and hence the conventional methods for the constant diffusion are inapplicable to \eqref{eq-1} and new ideas are much in demand. The existing results in the literature are only confined with a specific case $\alpha=0$ in \eqref{relation}, which implies that $\phi(v) = -\gamma'(v)$ and hence reduces the system to
\begin{equation}\label{eq-4}
 \left\{\begin{array}{ll}
  u_t=\Delta(\gamma(v)u),\quad &
x\in \Omega, t>0,\\
 \disp{ \tau v_t=d\Delta v +u- v},\quad &
x\in \Omega, t>0.\\
\partial_\nu u=\partial_\nu v=0,\quad &
x\in \partial\Omega, t>0,\\
\disp{u(x,0)=u_0(x)}, \ v(x,0)= v_0(x),\quad  &
x\in \Omega.
 \end{array}\right.
\end{equation}
If $\gamma(v)=\frac{c_0}{v^k}$ (i.e. algebraically decreasing) with $k>0$, it was proved in \cite{Yoon-1} that global bounded solutions exist in any dimensions provided that $c_0>0$ is small enough. The global existence result was later extended to the parabolic-elliptic case model \eqref{eq-4} (i.e. $\tau=0$) in \cite{Ahn-1} for any $0<k<\frac{2}{n-2}$ and $c_0>0$. Recently the global existence of weak solutions of \eqref{eq-4} with $\tau=1$ with large initial data was established in \cite{Kim-NARWA-2019}. When $\gamma(v)=e^{-\chi v}$ (exponentially decreasing), the size of initial cell mass is crucial for the global dynamics. A critical mass phenomenon, different from the case of algebraically decreasing $\gamma(v)$,  was identified in \cite{Jin-3} in two dimensions:  if $n=2$, there is a critical number $m=4\pi/\chi>0$ such that the solution of \eqref{eq-4} with $\tau=d=1$ may blow up if the initial cell mass $\|u_0\|_{L^1(\Omega)}>m$ while global bounded solutions exist if $\|u_0\|_{L^1(\Omega)}<m$. This result was further refined in \cite{FujieJiang2020-2} showing that the blowup occurs at the infinity time for $\tau=0$ and in \cite{FujieJiang2020-1} for $\tau=1$. In \cite{Burger}, the global existence of weak solutions for arbitrary large masses and space dimension was proved and infinite-time blowup was also shown by minimizing the entropy.

To describe the stripe pattern formation observed in the experiment of \cite{Liu-1},  a logistic source was added into the first equation of \dref{eq-4} in \cite{FuFuFugttt}
\begin{equation}\label{eq-5}
 \left\{\begin{array}{ll}
  u_t=\Delta(\gamma(v) u)+\mu u(1-u),\quad &
x\in \Omega, t>0,\\
 \disp{ \tau v_t=d\Delta v +u- v},\quad &
x\in \Omega, t>0,\\
 \end{array}\right.
\end{equation}
where $\gamma' (v) < 0$ and $\mu>0$ denotes the intrinsic cell growth rate. The system  \eqref{eq-5}  was called density-suppressed motility model in \cite{FuFuFugttt}.
It was shown in \cite{Jin-1} that the system \eqref{eq-5} has a unique global classical solution in two dimensional spaces if $\gamma(v)$ satisfies the following: $\gamma(v)\in C^3([0,\infty)), \gamma(v)>0\ \  \mathrm{and}~\gamma'(v)<0 \ \ \mathrm{on} ~[0,\infty)$, $\lim\limits_{v \to \infty}\gamma(v)=0$ and $\lim\limits_{v \to \infty}\frac{\gamma'(v)}{\gamma(v)}$ exists. Moreover, the constant steady state $(1,1)$ of \eqref{eq-5} is proved to be globally asymptotically stable if
$
\mu>\frac{K_0}{16}$ where $ K_0=\max\limits_{0\leq v \leq \infty}\frac{|\gamma'(v)|^2}{\gamma(v)}$. Recently, similar results have been extended to higher dimensions ($n\geq 3$) for large $\mu>0$ in \cite{Wang-1} and for weakers conditions on $\gamma(v)$ in \cite{FujieJiang2020-1, JW-DCDSB-2020}. On the other hand, for small $\mu>0$, the existence/{\color{black}nonexistence} of nonconstant steady states  of \eqref{eq-5} was rigorously established under some constraints on the parameters in \cite{Ma-1} and the periodic pulsating wave is analytically obtained by the multi-scale analysis.    When $\gamma(v)$ is a piecewise constant function, the dynamics of discontinuity interface was studied in \cite{Smith-1} and the existence of discontinuous traveling wave solution was established in \cite{Lui}. When $\gamma(v)$ is smooth, the existence of smooth traveling wavefront solutions was proved in \cite{Li-Wang-2020}.

From the results recalled above, we see that when $\gamma(v)$ and $\phi(v)$ satisfy the relation \eqref{relation}, the existing results are all confined to the case $\alpha=0$ and no results are available to the case $\alpha\ne 0$. As a first attempt, the parabolic-elliptic system \dref{eq-1} with $\tau= 0$ was recently considered in \cite{Wang-2} by attaching the following hypotheses on $\gamma(v)$ and $\phi(v)$:

\begin{itemize}
\item[(L1)] $\gamma(v) \in C^2([0,\infty))$ and $\gamma(v)> 0$ for all $v \in [0,\infty)$.
\vspace{0.1cm}

\item[(L2)]
\begin{enumerate}
\item[(a)] $\phi(v) \in C^2([0,\infty))$, $\phi(v)\geq 0$ and $\phi'(v)<0$ for $v\in [0,\infty)$;
\item[(b)] $\lim\limits_{v \to \infty}v\phi(v)<\infty$ if $n>3$.
\end{enumerate}
\vspace{0.1cm}

\item[(L3)] $\inf\limits_{v\geq 0} \disp{\frac{\gamma(v) |\phi'(v)|}{|\phi(v)|^2}>\frac{n}{2}}$.
\end{itemize}

Then it was shown \cite{Wang-2} that if (L1)-(L3) are satisfied and the solution satisfies an additional constraint
\begin{equation}\label{extra}
\int_\Omega \phi^{-p}(v)dx < \infty~~~\mbox{for some}~~~~ p >\frac{N}{2},
\end{equation}
then the system \eqref{eq-1} with $\tau=0$ admits a unique global classical solution with uniform-in-time bound for given initial data $(u_0, v_0)\in [W^{1,\infty}(\Omega)]^2$. Moreover when $\gamma(v)$ and $\phi(v)$ are algebraically or exponentially decreasing functions, the conditions ensuring the global boundedness are specified.

The purpose of this paper is to investigate the fully parabolic system \eqref{eq-1} (i.e. $\tau=1$) and find conditions warranting the global boundedness of solutions. The proof in \cite{Wang-2} fully made use of the elliptic structure of the second equation of \eqref{eq-1} to construct a positive definite quadratic form to derive the essential $L^p$-estimates ($p>N$). But this idea was not applicable to the parabolic system \eqref{eq-1} with $\tau=1$.
The key step in this paper is to  build a uniform estimate for the weighted integral $\int_\Omega u^{p} v^{-q} dx$ with some $p,q > 0$ where the exponent $q$ depends on the unknown $v$, which develops the ideas in the literature (cf.  \cite{Fujie-5, Winkler-5}) where the exponent is a constant determined by a simple quadratic equation. To this end, we derive  an auxiliary lemma (i.e. Lemma \ref{Lmma-1}) to elucidate the complexity involved in this $v$-dependent exponent. Then we use the estimate for the weighted integral $\int_\Omega u^{p} v^{-q} dx$ to derive the $W^{1,p}$-estimate for $v$ with $p>N$ which yields the $L^\infty$-bound of $v$  and hence rules out the diffusion degeneracy. Finally we employ the Moser iteration to establish the global boundeness of solutions.

The hypotheses attached to the motility functions $\gamma(v)$ and $\phi(v)$ are the following
\begin{quote}
\begin{itemize}
\item[] ${(H_1)}$ $\gamma(v)\in C^2 ([0,\infty))$ and $\gamma(v) > 0$ for all $v\in[0,\infty)$.
\item[]${(H_2)}$ $\phi(v)\in C^2 ([0,\infty))$ and $\phi(v)\geq 0$ for all $v\in[0,\infty)$.
\item[] ${(H_3)}$ $\gamma(v)$ and $\phi(v)$ fulfills
$$\inf_{v\geq0}\frac{d\gamma(v)}{v\phi(v)(v\phi(v)+d-\gamma(v))_+}>\frac{N}{2}.$$
\end{itemize}
\end{quote}
For the initial data $(u_0,v_0)$, we assume that
\begin{equation}\label{eq-6}
\displaystyle{(u_0, v_0)\in [W^{1,\infty}(\Omega)]^2~~\mbox{with}~~u_0, v_0\geq0~~\mbox{in}~~\Omega~~\mbox{and}~~u_0, v_0\not\equiv0}.
\end{equation}
Then our main results are stated in the following theorem.
\begin{theorem}\label{theorem3}
 Let $\Omega\subset\mathbb{R}^N(2\leq N\leq 4)$ be a bounded  domain with smooth boundary and initial data satisfy \dref{eq-6}.
If  $\gamma$ and $\phi$ satisfy hypotheses $(H_1)$-$(H_3)$, then problem \eqref{eq-1} has a unique global classical solution $(u, v) \in [u\in C^0(\bar{\Omega}\times[0,\infty))\cap C^{2,1}(\bar{\Omega}\times(0,\infty))]^2$ satisfying
 $$
\|u(\cdot, t)\|_{L^\infty(\Omega)}+\| v(\cdot, t)\|_{W^{1,\infty}(\Omega)}\leq C~~ \mbox{for all}~~ t\in(0,\infty),
$$
where $C>0$ is a constant $C$ independent of $t$.
\end{theorem}

\begin{remark} There are two remarks concerning the above results.

\begin{enumerate}
\item If $d=\gamma(v)=1$ and $\phi(v)=\frac{\chi}{v}$, then the hypothesis $(H_3)$    is equivalent to  

$$\inf_{v\geq0}\frac{d\gamma(v)}{v\phi(v)(v\phi(v)+d-\gamma(v))_+}=\frac{1}{\chi^2}>\frac{N}{2}.
$$
Thus the results in Theorem \ref{theorem3} says that the system \eqref{eq-1} with $\tau=1$ admit globally bounded solutions if $\chi<\sqrt{\frac{2}{N}}$, which uncovers the result of Fujie (\cite{Fujie-5}).
\item In \cite{Wang-2}, two essential conditions (L3) and \eqref{extra} are imposed to guarantee the global boundedness of solutions to \eqref{eq-1} with $\tau=0$, as recalled above, where the condition \eqref{extra} imposed a sole restriction on $\phi(v)$.   In this paper, we do not have such a condition and only hypothesis $(H_3)$ is prescribed.
\end{enumerate}
\end{remark}
If $\gamma(v)$ and $\phi(v)$ are explicitly given, the condition   $(H_3)$  can be
specified. In the following theorem, as an application of Theorem \ref{theorem3}, we consider algebraically decreasing  motility
functions $\gamma(v)$ and $\phi(v)$ connected by the proportionality relation \eqref{relation} in the original Keller-Segel model and specify the conditions warranting the global boundedness.

Before stating the result, we recall the following result (see \cite[Lemma 2.2]{Fujie-5}) asserting that  the solution component $v$ has a positive lower bound.
\begin{lemma}\label{lemm4-02}
Let $(u,v)$ be a solution to \dref{eq-1}.
Then there exists a positive constant $\eta$ independent of $t$
such that
\begin{equation}\label{eq5-01}
\inf\limits_{x \in \Omega}v(x,t)\geq \eta~~\mbox{for all}~~t>0.
\end{equation}
\end{lemma}

Then the following results hold.

\begin{theorem}\label{theorem4}
Let $\Omega\subset\mathbb{R}^N(2\leq N\leq 4)$ be a bounded  domain with smooth boundary and initial data satisfy \dref{eq-6}.
 If $\gamma(v)$ and $\phi(v)$ satisfy the relation \eqref{relation} with $0<\alpha <1$
 and
$\gamma(v) =\frac{\sigma}{v^\lambda}$ ($\sigma > 0,\lambda>0$)
 such that
\begin{equation}
\frac{N}{2}<
\begin{cases}
\displaystyle \frac{d}{\lambda(1-\alpha)\Big(\frac{[\lambda(1-\alpha)-1]\sigma}{\eta^\lambda}+d\Big)_+}, & if \ \ \lambda>\frac{1}{1-\alpha}\\
\displaystyle \frac{1}{\lambda(1-\alpha)}, & if \ \  \lambda\leq\frac{1}{1-\alpha}
\end{cases}
\end{equation}
where $\eta$ is defined in Lemma \ref{lemm4-02}, then the  system \dref{eq-1} has a unique classical solution $(u,v)$ in $\bar{\Omega}\times[0,\infty)$, which is uniformly bounded in time such that
 $$
\|u(\cdot, t)\|_{L^\infty(\Omega)}+\| v(\cdot, t)\|_{W^{1,\infty}(\Omega)}\leq C~~ \mbox{for all}~~ t\in(0,\infty)
$$
with some positive constant $C$ independent of $t$.
\end{theorem}

\section{Preparatory results}
In this section, we will present some preparatory results for proving the global boundedness of solutions to \eqref{eq-1}, including local existence of classical solutions, some frequently used inequalities and basic estimates on the solutions.   In what follows, without confusion, we shall abbreviate $\int_\Omega fdx$ as
$\int_\Omega f$ for simplicity. Moreover, we shall use $C_i (i = 1,2,3,\ldots)$ to denote a generic constant
which may vary in the context.

\subsection{Local existence} We first state the local-in-time existence of classical solutions of the problem \dref{eq-1}
\begin{lemma}[Local existence]\label{lemma70}
 Let $\Omega\subset\mathbb{R}^N(N\geq1)$ be a bounded domain with smooth boundary and assume $\gamma(v)$ and $\phi(v)$ satisfy the hypotheses $(H_1)$ and $(H_2)$.
If the initial data  $(u_0, v_0)$ satisfy \dref{eq-6}, then there there is a $T_{max}\in (0, \infty]$ such that the problem \dref{eq-1} with $\tau=1$ admit a unique classical solution $\disp{(u,v)\in [C^0(\bar{\Omega}\times[0,T_{max}))\cap C^{2,1}(\bar{\Omega}\times(0,T_{max}))]^2}$ satisfying $u,v>0$
  in $\Omega\times(0,T_{max})$.
%
Moreover, if  $T_{max}<+\infty$, then
$$
\|u(\cdot, t)\|_{L^\infty(\Omega)}+\| v(\cdot, t)\|_{W^{1,\infty}(\Omega)}\rightarrow\infty~~ \mbox{as}~~ t\nearrow T_{max}.
$$
\end{lemma}

The local existence results stated in Lemma \ref{lemma70} can be proved by the standard Schauder fixed point theorem along with parabolic regularity theory (cf.  see \cite{Jin-1, Zheng0}) or by the Amen's theorem \cite{Amann-book} (cf. \cite{Jin-2} for details). Therefore, we omit the proof for brevity.

\subsection{Some inequalities}

The following well-known inequality will be frequently used.
 \begin{lemma}(\cite{Nirenber-1})\label{lemma4sddd1ffgg}
Let $\Omega$
 be a bounded Lipschitz domain in $R^N$, $p, q, r, s \geq 1, j,m \in \mathbb{N}_0$ and $\alpha\in [\frac{j}{m}, 1]$
satisfying $\frac{1}{p} = \frac{j}{m} + (\frac{1}{r}-\frac{m}{N})\alpha+ \frac{1-\alpha}{q}$. Then there is a positive constant $C$ such that for all
functions $w\in W^{m,r}(\Omega)\cap L^s(\Omega)$, the following inequality holds
$$\| D^jw\|_{L^{p}(\Omega)} \leq C (\|D^mw\|_{L^{r}(\Omega)}^{\alpha}\|w\|^{1-\alpha}_{L^q(\Omega)}+\|w\|_{L^{s}(\Omega)}).$$
\end{lemma}

\bigbreak

\begin{lemma}(\cite{Lankeit-11})\label{lemma-7}
Let $y\in C^1((0, T )) \cap C^0([0, T ))$ with some $T\in(0,\infty]$,  $B > 0, A > 0$ and the nonnegative function $h \in C^0([0, T ))$ satisfy
\begin{equation}\label{eq-010}
\begin{array}{ll}
\displaystyle{
 y'(t)+Ay(t)\leq h(t)~~~\mbox{and}~~~\int_{(t-1)_+}^{t}h(s)ds\leq B ~~\mbox{for a.e.}~~t\in(0,T)}.
\end{array}
\end{equation}
Then it follows that
$$y(t)\leq y_0+\frac{B}{1-e^{-A}}~~\mbox{for all}~~t\in(0,T).$$
\end{lemma}

The  following estimates of $v$ depending on the bound of $u$ is useful to our later analysis.
\begin{lemma}\label{lemma-014} (\cite{Fujie-5,Winkler-5})
Let $1\leq \theta,\mu \leq\infty$. Then there is a constant $C>0$ independent of $t$ such that the solution of \eqref{eq-1} satisfies the following estimates:

(i) If $\frac{N}{2}(\frac{1}{\theta}-\frac{1}{\mu})<1,$ then
\begin{equation*}
\|v(\cdot,t)\|_{L^\mu(\Omega)}\leq C(1+\sup_{s\in(0,\infty)}\|u(\cdot,s)\|_{L^\theta(\Omega)})~~~\mbox{for all}~~~ t>0.
\label{eq3-1}
\end{equation*}

(ii) If $\frac{1}{2}+\frac{N}{2}(\frac{1}{\theta}-\frac{1}{\mu})<1,$ then 
\begin{equation*}
\|\nabla v(\cdot,t)\|_{L^\mu(\Omega)}\leq C(1+\sup_{s\in(0,\infty)}\|u(\cdot,s)\|_{L^\theta(\Omega)})~~~\mbox{for all}~~~ t>0.
\label{eq3-2}
\end{equation*}
\end{lemma}

\subsection{Some basic estimates on solutions}
Below we shall show some basic but important preparatory estimates on solutions in order to prove our results.

\begin{lemma}\label{lemma-010}
Let $(u,v)$ be a solution to \dref{eq-1} obtained in Lemma \ref{lemma70} with a maximal time $T_{max}>0$. The it follows that
\begin{equation}\label{eq-012}
\int_{\Omega}u= \int_{\Omega}{u_{0}}~~\mbox{for all}~~ t\in(0, T_{max})
\end{equation}
and
\begin{equation}\label{123}
\int_{\Omega}{v}\leq \max\bigg\{\int_{\Omega}{u_{0}},\int_{\Omega}{v_{0}}\bigg\}~~\mbox{for all}~~ t\in(0, T_{max}).
\end{equation}
\end{lemma}
\begin{proof}
Integration of the first equation of \eqref{eq-1} along with the Neumann boundary conditions immediately gives \eqref{eq-012} and integration of the second equation of \eqref{eq-1} with \eqref{eq-012} yields \eqref{123}.
\end{proof}

The following lemma will serve as a source for several integral estimates in the sequel. It will be applied with certain parameters $p > \frac{N}{2}$ to provide global smooth solutions for system \dref{eq-1}. We shall prepare some useful estimates. Firstly,  we recall and  build some uniform estimate for the weighted integral
$\int_\Omega u^p v^{-q} dx$ with some $p>0,q > 0$, which may depend on $v$.
\begin{lemma}\label{lemma3-1} Let $\Omega\subset\mathbb{R}^N(N\geq1)$ be a smooth bounded  domain.
    Let $(u,v)$ be a solution to \dref{eq-1} obtained in Lemma \ref{lemma70} with a maximal time $T_{max}>0$.  Then
for all $\tilde{p}, \tilde{q} \in \mathbb{R}$, 
we have
\begin{equation}
\begin{array}{rl}
&\disp{\frac{d}{dt}\int_{\Omega}u^{\tilde{p}}v^{\tilde{q}}}
\\
=&\disp{-\tilde{p}(\tilde{p}-1)\int_{\Omega}u^{\tilde{p}-2}v^{\tilde{q}}\gamma(v)|\nabla u|^2+\int_{\Omega} u^{\tilde{p}}v^{\tilde{q}-2}[\tilde{p}\tilde{q}v\phi(v)-d\tilde{q}(\tilde{q}-1)]|\nabla v|^2}\\
&\disp{+\int_{\Omega}\tilde{p}u^{\tilde{p}-1}v^{\tilde{p}-1}[(\tilde{p}-1)v\phi(v)-\tilde{q}\gamma(v)-d\tilde{q}]\nabla u\cdot\nabla v}
\\
&+\disp{\tilde{q}\int_{\Omega}u^{\tilde{p}}v^{\tilde{q}-1}(u-v)~~ \mbox{for all}~~  t\in(0,T_{max}).}
\\
\end{array}
\label{eq4-1}
\end{equation}
\end{lemma}
\begin{proof}
A direct calculation resting on integration by parts and the Neumann boundary
conditions yields
\begin{eqnarray}\label{I}
\begin{aligned}
&\disp{\frac{d}{dt}\int_{\Omega}u^{\tilde{p}}v^{\tilde{q}}}
\\
=&\disp{\tilde{p}\int_{\Omega}u^{\tilde{p}-1}v^{\tilde{q}}u_t+\tilde{q}\int_{\Omega}u^{\tilde{p}}v^{\tilde{q}-1}v_t}
\\
=&\disp{\tilde{p}\int_{\Omega}u^{\tilde{p}-1}v^{\tilde{q}}[\nabla\cdot(\gamma(v)\nabla u-u\phi(v)\nabla v)]+\tilde{q}\int_{\Omega}u^{\tilde{p}}v^{\tilde{q}-1}(d\Delta v+u-v)}
\\
=&\disp{-\tilde{p}\int_{\Omega}[\gamma(v)\nabla u-u\phi(v)\nabla v)]\cdot\nabla(u^{\tilde{p}-1}v^{\tilde{q}})-\tilde{q}d\int_{\Omega}\nabla (u^{\tilde{p}}v^{\tilde{q}-1})\cdot\nabla v}+\disp{\tilde{q}\int_{\Omega}u^{\tilde{p}}v^{\tilde{q}-1}(u-v)}\\
=&I_1+I_2+I_3.
\end{aligned}
\end{eqnarray}
With further simple calculations, we have
\begin{eqnarray}\label{I1}
\begin{aligned}
I_1=&\disp{-\tilde{p}(\tilde{p}-1)\int_{\Omega}u^{\tilde{p}-2}v^{\tilde{q}}[\gamma(v)\nabla u-u\phi(v)\nabla v)]\cdot\nabla u} \\
&\disp{-\tilde{p}\tilde{q}\int_{\Omega}u^{\tilde{p}-1}v^{\tilde{q}-1}[\gamma(v)\nabla u-u\phi(v)\nabla v)]\cdot\nabla v}
\end{aligned}
\end{eqnarray}
and
\begin{eqnarray}\label{I2}
\begin{aligned}
I_2=\disp{-d\tilde{q}\tilde{p}\int_{\Omega}u^{\tilde{p}-1}v^{\tilde{q}-1}\nabla u\cdot\nabla v-d\tilde{q}(\tilde{q}-1)\int_{\Omega} u^{\tilde{p}}v^{\tilde{q}-2}|\nabla v|^2}.
\end{aligned}
\end{eqnarray}
Substituting \eqref{I1}-\eqref{I2} into \eqref{I} and regrouping the same-like terms, we get \eqref{eq4-1} immediately.
\end{proof}

In  order to cope with the barrier caused by the  non-constant $\gamma(v)$ and $\phi(v)$, we prove the following lemma which plays a key role in obtaining  the $L^p$-estimate for $u$ for some $p >\frac{N}{2}$. In the sequel, for notational brevity, we set
\begin{equation}\bar{\phi}(v)=v\phi(v).
\label{eq-10}
\end{equation}

\begin{lemma}\label{Lmma-1}
Let $\gamma$ and $\phi$ satisfy hypotheses $(H_1)$-$(H_3)$. Let $p>\frac{N}{2} (N\geq 2)$ and close to $\frac{N}{2}$. Set
\begin{equation}\label{ABC}
\begin{cases}
A(v)
=4\gamma(v)d+p\gamma^2(v)+pd^2-2dp\gamma(v),\\
B(v)=2(p-1)[2\gamma(v)d+p\bar{\phi}(v)(\gamma(v)-d)],\\
C(v)=p(p-1)^2\bar{\phi}^2(v).
\end{cases}
\end{equation}
 Then $A(v), B(v)$ and $C(v)$ are all positive such that
 \begin{equation}\begin{array}{rl}
\bigg(\disp\frac{2C(v)}{B(v)},\disp\frac{B(v)}{2A(v)}\bigg]\cap\bigg(0,\min\Big\{\frac{N}{2},p\Big\}\bigg)
\neq&\emptyset~~ \mbox{for all}~~  v\geq0.\\
\end{array}\label{eq5-07}
\end{equation}
\end{lemma}
\begin{proof}
First it is easy to verify that $A(v)$ and $C(v)$ are all positive.
In view of  $(H_1)$ and
$(H_3)$, one can pick $p>\frac{N}{2} (N\geq 2)$ which is close to $\frac{N}{2}$ such that for any  $v\geq0$,
\begin{equation}\frac{d\gamma(v)}{\bar{\phi}(v)(\bar{\phi}(v)+d-\gamma(v))_+}>p.
\label{eq5-09}
\end{equation}
Then it is clear that
\begin{equation}
\begin{array}{rl}
&\disp
\disp\gamma(v)>\frac{p\bar{\phi}(v)[d+\bar{\phi}(v)]}{d+p\bar{\phi}(v)}=:\gamma_1(v)~~~\mbox{for all}~~v\geq0.\\
\end{array}
\label{eq5-10}
\end{equation}
To proceed, we define
\begin{equation}
\begin{cases}
\gamma_2(v)=\disp\frac{p\bar{\phi}(v)d}{2d+p\bar{\phi}(v)},\\
\gamma_3(v)=\disp\frac{p\bar{\phi}(v)[2(p-1)\bar{\phi}(v)+dN]}{N(2d+p\bar{\phi}(v))},\\
\gamma_4(v)=\disp\frac{\bar{\phi}(v)[(p-1)\bar{\phi}(v)+pd]}{2d+p\bar{\phi}(v)}.
\end{cases}
\end{equation}
Then choosing  $p>\frac{N}{2}$ but close to $\frac{N}{2}$ and noticing that $\bar{\phi}(v)\geq0$, we can verify that
\begin{equation}\begin{array}{rl}
\gamma_1(v)-\gamma_4(v)=&\disp{\frac{p\bar{\phi}(v)[d+\bar{\phi}(v)]}{d+p\bar{\phi}(v)}-\frac{\bar{\phi}(v)[(p-1)\bar{\phi}(v)+pd]}{2d+p\bar{\phi}(v)}}\\
>&\disp{\frac{\bar{\phi}(v)[p\bar{\phi}(v)+pd]}{2d+p\bar{\phi}(v)}-\frac{\bar{\phi}(v)[(p-1)\bar{\phi}(v)+pd]}{2d+p\bar{\phi}(v)}}\\
\geq&\disp{\frac{\bar{\phi}(v)[(p-1)\bar{\phi}(v)+pd]}{2d+p\bar{\phi}(v)}-\frac{\bar{\phi}(v)[(p-1)\bar{\phi}(v)+pd]}{2d+p\bar{\phi}(v)}}=0\\
\end{array}
\label{eq5-15}
\end{equation}
as well as
\begin{equation}\begin{array}{rl}
\gamma_3(v)-\gamma_2(v)=&\disp{\frac{p\bar{\phi}(v)[2(p-1)\bar{\phi}(v)+dN]}{N(2d+p\bar{\phi}(v))}-\frac{p\bar{\phi}(v)d}{2d+p\bar{\phi}(v)}}\\
=&\disp{p\bar{\phi}(v)
\frac{2(p-1)\bar{\phi}(v)}{N(2d+p\bar{\phi}(v))}}\geq0\\
\end{array}
\label{eq5-16}
\end{equation}
and
\begin{equation}\begin{array}{rl}
\gamma_1(v)-\gamma_3(v)=&\disp{\frac{p\bar{\phi}(v)[d+\bar{\phi}(v)]}{d+p\bar{\phi}(v)}-\frac{p\bar{\phi}(v)[2(p-1)\bar{\phi}(v)+dN]}{N(2d+p\bar{\phi}(v))}}\\
=&\disp{p\bar{\phi}(v)
\frac{[p\bar{\phi}^2(v)(N-2p+2)+\bar{\phi}(v)d(2N-2p+2)+Nd^2]}{N[d+p\bar{\phi}(v)][2d+p\bar{\phi}(v)]}}>0.\\
\end{array}
\label{eq5-17}
\end{equation}
Collecting \dref{eq5-15}--\dref{eq5-17} and using \dref{eq5-10},  one gets
\begin{equation}\label{gamma}
\gamma(v)>\gamma_1(v)\geq \max\{\gamma_2(v), \gamma_3(v), \gamma_4(v)\}.
\end{equation}
Since
\begin{equation}
\begin{array}{rl}
&B(v)=2(p-1)[2\gamma(v)d+p\bar{\phi}(v)(\gamma(v)-d)]>0~\Leftrightarrow~\disp\gamma(v)>\gamma_2(v),
\end{array}
\label{eq5-12}
\end{equation}
therefore $B(v)>0$ is ensured due to \eqref{gamma}. With simple calculations, one has
$$
\begin{array}{rl}
&B^2(v)-4A(v)C(v)\\
=&4(p-1)^2\{[2\gamma(v)d+p\bar{\phi}(v)(\gamma(v)-d)]^2-p\bar{\phi}^2(v)[4\gamma(v)d+p(\gamma(v)-d)^2]\}\\
=&16(p-1)^2\gamma(v)d\{d\gamma(v)+p\gamma(v)\bar{\phi}(v)-dp\bar{\phi}(v)-p\bar{\phi}^2(v)\},
\end{array}
$$
and hence
\begin{equation}
\begin{array}{rl}
&\disp \frac{2C(v)}{B(v)}<\frac{B(v)}{2A(v)}~\Leftrightarrow~~B^2(v)-4A(v)C(v)>0~\Leftrightarrow~
\disp\gamma(v)>\gamma_1(v).
\end{array}
\label{eq5-11}
\end{equation}
Furthermore with $p>\frac{N}{2}$, one can check
\begin{equation}
\begin{array}{rl}
\disp\frac{2C(v)}{B(v)}:=&\disp\frac{(p-1)p\bar{\phi}^2(v)}{2\gamma(v)d+p\bar{\phi}(v)(\gamma(v)-d)}<\frac{N}{2}~~\Leftrightarrow~~
\disp\gamma(v)>\gamma_3(v)
\end{array}
\label{eq5-13}
\end{equation}
and
\begin{equation}
\begin{array}{rl}
\disp\frac{2C(v)}{B(v)}:=&\disp\frac{(p-1)p\bar{\phi}^2(v)}{2\gamma(v)d+p\bar{\phi}(v)(\gamma(v)-d)}<p~~\Leftrightarrow~~
\disp\gamma(v)>\gamma_4(v).
\end{array}
\label{eq5-14}
\end{equation}
Finally  with \dref{gamma}--\dref{eq5-14}, we get \eqref{eq5-07} and hence completes  the proof.
\end{proof}

To prove our main results, we need to derive some {\it a priori} estimates of $u$.
Making use of Lemma \ref{Lmma-1}, instead of working on the $L^p$-estimates of $u$,  we first
derive  a bound on the weighted $L^p$-estimates of $u$, namely $\int_{\Omega}u^{p}v^{-q}$, with  some
positive  exponents $p$ and $q$. To this end, we first establish a differential inequality on the weighted $L^p$-norm of $u$ as follows.

\begin{lemma}\label{lem5-03} Let $\Omega\subset\mathbb{R}^N(N\geq2)$ be a smooth bounded  domain.
Assume that  $\gamma(v) $ and
$\phi(v)$ satisfy hypotheses $(H_1)$-$(H_3)$.
 Let  $(u,v)$ be a solution to \dref{eq-1} in a maximal time interval $(0,T_{max})$. Then for any $\bar{p}>\frac{N}{2}$ which is close to $\frac{N}{2}$ and $p\in(1,\bar{p}]$ , there exists
 $q\in(\disp\frac{2C(v)}{B(v)},\disp\frac{B(v)}{2A(v)}]\cap(0,\min\{\frac{N}{2},p\})$ such that 
\begin{equation}
\begin{array}{rl}
&\disp{\frac{d}{dt}\int_{\Omega}u^{p}v^{-q}\leq q\int_{\Omega}u^{p}v^{-q}- q\int_{\Omega}u^{p+1}v^{-q-1}~~ \mbox{for all}~~  t\in(0,T_{max}),}
\\
\end{array}
\label{eq6-01}
\end{equation}
where $A(v)$ as well as  $B(v)$ and $C(v)$ are the same as  \dref{ABC}.
\end{lemma}
\begin{proof}
First,  from the proof of Lemma \ref{Lmma-1}, see \eqref{eq5-10},
 there exists a positive constant $\bar{p}>\frac{N}{2}$ which is close to  $\frac{N}{2}$ such that
\begin{equation}\disp \gamma(v)>\frac{\bar{p}\bar{\phi}(v)[d+\bar{\phi}(v)]}{d+\bar{p}\bar{\phi}(v)}
\label{eq6-02}
\end{equation}
for any $v\geq0$.
Observing that $\frac{\bar{p}\bar{\phi}(v)[d+\bar{\phi}(v)]}{d+\bar{p}\bar{\phi}(v)}$  is monotonically increasing  with respect
to $\bar{p}>0$.
Thus, by \dref{eq6-02}, we conclude that for all $p\in(1,\bar{p}]$,
\begin{equation*}
\disp \gamma(v)>\frac{p\bar{\phi}(v)[d+\bar{\phi}(v)]}{d+p\bar{\phi}(v)}
\label{eq6-03}
\end{equation*}
for any $v\geq0$.

Next, for  $p\in(1,\bar{p}],$ choosing $\tilde{p}:=p>1$ and $\tilde{q}:=-q$  in Lemma \ref{lemma3-1}, we obtain
\begin{equation}
\begin{array}{rl}
&\disp{\frac{d}{dt}\int_{\Omega}u^{p}v^{-q}}
\\[3mm]
=&\disp{-p(p-1)\int_{\Omega}u^{p-2}v^{-q}\gamma(v)|\nabla u|^2-q\int_{\Omega}u^{p}v^{-q-2} [pv\phi(v)+d(q+1)]|\nabla v|^2}\\[3mm]
&\disp{+p\int_{\Omega}u^{p-1}v^{-q-1}[(p-1)v\phi(v)+q\gamma(v)+dq]\nabla u\cdot\nabla v}
\\
&\disp{-q\int_{\Omega}u^{p+1}v^{-q-1}+q\int_{\Omega}u^{p}v^{-q}~~ \mbox{for all}~~  t\in(0,T_{max}),}
\\
\end{array}
\label{eq6-04}
\end{equation}
where by the  Young inequality, it follows that
\begin{equation}
\begin{array}{rl}
&\disp{p\int_{\Omega}u^{p-1}v^{-q-1}[(p-1)v\phi(v)+q\gamma(v)+dq]\nabla u\cdot\nabla v}
\\
\leq&\disp{p(p-1)\int_{\Omega}u^{p-2}v^{-q}\gamma(v)|\nabla u|^2}\\
& \ \  \disp +\frac{p}{4(p-1)}\int_{\Omega}\frac{[(p-1)v\phi(v)+q\gamma(v)+dq]^2}{\gamma(v)}u^{p}v^{-q-2}|\nabla v|^2
\end{array}
\label{eq6-05}
\end{equation}
for all $t\in(0,T_{max}).$
Inserting \dref{eq6-05} into \dref{eq6-04} yields that
\begin{equation}
\begin{array}{rl}
&\disp{\frac{d}{dt}\int_{\Omega}u^{p}v^{-q}}
\\[3mm]
\leq&\disp{\int_{\Omega}\bigg\{\frac{p}{4(p-1)}\times\frac{[(p-1)v\phi(v)+q\gamma(v)+dq]^2}{\gamma(v)}-dq(q+1)-pqv\phi(v)\bigg\}u^{p}v^{-q-2}|\nabla v|^2}
\\[3mm]
&+\disp{q\int_{\Omega}u^{p}v^{-q}- q\int_{\Omega}u^{p+1}v^{-q-1}~~ \mbox{for all}~~  t\in(0,T_{max}).}
\\
\end{array}
\label{eq6-06}
\end{equation}
 Denote
\begin{equation}g(p;q,v):=\frac{p}{4(p-1)}\times\frac{[(p-1)\bar{\phi}(v)+q\gamma(v)+dq]^2}{\gamma(v)}-dq(q+1)-pq\bar{\phi}(v).
\label{eq6-07}
\end{equation}
Then from \eqref{eq6-07}, we can get a quadric expression for $q$ as follows
\begin{equation}
\begin{array}{rl}
&4(p-1)\gamma(v)g(p;q,v)\\
=&[p(p-1)^2\bar{\phi}^2(v)+pq^2\gamma^2(v)+d^2pq^2+2p(p-1)q\bar{\phi}(v)\gamma(v)+2dp(p-1)q\bar{\phi}(v)+2dpq^2\gamma(v)]\\
&-4(p-1)\gamma(v)dq(q+1)-4(p-1)\gamma(v)pq\bar{\phi}(v)\\
=&p(p-1)^2\bar{\phi}^2(v)+pq^2\gamma^2(v)+d^2pq^2-2p(p-1)q\bar{\phi}(v)\gamma(v)+2dp(p-1)q\bar{\phi}(v)+2dpq^2\gamma(v)\\
&-4(p-1)\gamma(v)dq(q+1)\\
=&A(v)q^2-B(v)q+C(v),\\
\end{array}
\label{eq6-08}
\end{equation}
where
$A(v)$ as well as  $B(v)$ and $C(v)$ are given by \dref{ABC}.
Now, with Lemma \ref{Lmma-1},
one  can pick $q\in(\frac{2C(v)}{B(v)},\frac{B(v)}{2A(v)}]\cap(0,\min\{\frac{N}{2},p\})$, which ensures that
\begin{equation}\label{abc}
Aq^2\leq\frac{1}{2}Bq \ \ \mathrm{and}\ \ C<\frac{1}{2}Bq.
\end{equation}
With \eqref{abc}, one gets immediately from \eqref{eq6-08} that
$$g(p;q,v)< 0$$
for all $p\in(1,\bar{p}]$ and $v\geq0.$
Substituting the above inequality into \dref{eq6-06} and \dref{eq6-07}, we  get \dref{eq6-01}.
\end{proof}

\section{Proof of main results}
Now we have all necessary estimates to obtain weighted $L^p$-estimates for $u$ with some $p >\frac{N}{2}$, which
is a key step to derive the $L^p$-estimate for $u$ with $p>N$ that finally leads to the global boundedness of solutions.
\subsection{The {\it a priori} $L^p$-estimates}
We first develop a weighted $L^p$-norm of $u$.
\begin{lemma}\label{lemma-011} Let $\Omega\subset\mathbb{R}^N(2\leq N\leq 4)$ be a smooth bounded  domain.
Assume that  $\gamma(v) $ and
$\phi(v)$ satisfy hypotheses $(H_1)$-$(H_3)$.
 Let  $(u,v)$ be a solution to \dref{eq-1} in a maximal time inverval $(0,T_{max})$. Then
for  any $t \in(0,T_{max})$, there exist positive constants  $p>\frac{N}{2}$ as well as  $q>0$ such that
\begin{equation}
\begin{array}{rl}
&\disp{\int_{\Omega}u^{p}(\cdot,t)v^{-q}(\cdot,t)\leq C~~ \mbox{for all}~~  t\in(0,T_{max}),}
\end{array}
\label{eq7-01}
\end{equation}
where $C>0$ is a constant independent of $t$.
\end{lemma}
\begin{proof}We begin with \dref{eq6-01} and  apply the Young inequality to get
\begin{equation}
\begin{array}{rl}
&(q+1)\disp\int_{\Omega}u^{p}v^{-q}\\[3mm]
= &(q+1)\disp\int_{\Omega}u^{p}v^{-\frac{p(q+1)}{p+1}-[q-\frac{p(q+1)}{p+1}]}\\
\leq &\disp{q\int_{\Omega}u^{p\times\frac{p+1}{p}}v^{-\frac{p(q+1)}{p+1}\times\frac{p+1}{p}}+\frac{1}{p+1}(q \frac{(p+1)}{p})^{-p}(q+1)^{p+1} \int_{\Omega}v^{-[q-\frac{p(q+1)}{p+1}]\times(p+1)}}
\\[3mm]
= &\disp{q\int_{\Omega}u^{p+1}v^{-q-1}+\frac{1}{p+1}\bigg(q \frac{(p+1)}{p}\bigg)^{-p}(q+1)^{p+1} \int_{\Omega}v^{p-q},}
\end{array}
\label{eq7-02}
\end{equation}
which together with \dref{eq6-01} implies
\begin{equation}
\begin{array}{rl}
&\disp{\frac{d}{dt}\int_{\Omega}u^{p}v^{-q}+\int_{\Omega}u^{p}v^{-q}\leq \frac{1}{p+1}\bigg(\frac{q (p+1)}{p}\bigg)^{-p}(q+1)^{p+1} \int_{\Omega}v^{p-q}~~ \mbox{for all}~~  t\in(0,T_{max}).}
\\
\end{array}
\label{eq7-03}
\end{equation}
Therefore applying Lemma \ref{lemma-014} with the fact \dref{eq-012}, we get a constant $C_1 > 0$
such that
\begin{equation}
\begin{array}{rl}
\disp\int_{\Omega}v^{q_0}(\cdot,t)\leq &\disp{C_1~~~ \mbox{for all}~~  t\in(0,T_{max})~~~\mbox{and}~~~1\leq q_0<\frac{N}{N-2}.}
\\
\end{array}
\label{eq7-04}
\end{equation}
Due to $2\leq N\leq4$, it follows that $2\frac{N}{N-2} \in [2,\infty]$.  Hence for any $q>0$, one can choose $p>2$ but close to $2$ such that
$$0<p-q\leq q_0,$$
so that, employing the Young inequality, we derive from \eqref{eq7-04} that
\begin{equation}
\begin{array}{rl}
&\disp{\frac{1}{p+1}\bigg(\frac{q (p+1)}{p}\bigg)^{-p}(q+1)^{p+1} \int_{\Omega}v^{p-q}\leq C_2~~ \mbox{for all}~~  t\in(0,T_{max})}
\\
\end{array}
\label{eq7-05}
\end{equation}
with  some positive constant $C_2$.
For all $t\in(0,T_{max}),$ integrating \dref{eq7-03} from $0$ to $t$, taking into account Lemma \ref{lemma-7}, we get some positive constant $C_3$ such that
\begin{equation}
\begin{array}{rl}
\disp\int_{\Omega}u^{p}(\cdot,t)v^{-q}(\cdot,t)\leq &\disp{C_3~~ \mbox{for all}~~  t\in(0,T_{max}),}
\\
\end{array}
\label{eq7-06}
\end{equation}
which implies \dref{eq7-01} directly.

\end{proof}

\begin{lemma}\label{lem8-01}
Let $\Omega\subset\mathbb{R}^N(2\leq N\leq 4)$ be a smooth bounded  domain. Then for any $p>N$ there is a constant $C_0(p)>0$ independent of $t$ such that the solution of \dref{eq-1} satisfies
\begin{equation}\label{LP}
\|u(\cdot,t)\|_{L^p(\Omega)} \leq C_0(p)~~\mbox{for all}~~t\in(0,T_{max}).
\end{equation}
\end{lemma}

\begin{proof}
Firstly, according to Lemma  \ref{lemma-011}, we can pick $\kappa>\frac{N}{2}$ and $q_0\in (0,\frac{N}{2})$ such that 
\begin{equation}
\begin{array}{rl}
&\disp{\int_{\Omega}u^{\kappa}v^{-q_0}\leq C_1~~~\mbox{for all}~~t\in(0,T_{max})}
\\
\end{array}
\label{eq-802}
\end{equation}
holds with some $C_1>0.$ Since $q_0<\frac{N}{2}$ and $\kappa>\frac{N}{2}$, we may fix a number  $l_0\in(\frac{N}{2},\kappa)$ such that
$l_0<\frac{N(\kappa-q_0)}{N-2q_0}$.
Using \dref{eq-802} and H\"{o}lder inequality, we find that
\begin{equation}
\begin{array}{rl}
\left(\disp\int_{\Omega}u^{l_0}\right)^{\frac{1}{l_0}}\leq&\disp{\left(\int_{\Omega}u^{\kappa}v^{-q_0}\right)^{\frac{1}{\kappa}}
\left(\int_{\Omega}v^{\frac{l_0q_0}{\kappa-l_0}}\right)^{\frac{\kappa-l_0}{l_0\kappa}}}
\\
\leq&\disp{C_1^{\frac{1}{\kappa}}
\left(\int_{\Omega}v^{\frac{l_0q_0}{\kappa-l_0}}\right)^{\frac{\kappa-l_0}{l_0\kappa}}}
\\
=&\disp{C_1^{\frac{1}{\kappa}}
\|v(\cdot,t)\|_{L^{{\frac{l_0q_0}{\kappa-l_0}}}(\Omega)}^\frac{q_0}{\kappa}~~~\mbox{for all}~~t\in(0,T_{max}).}
\\
\end{array}
\label{eq-803}
\end{equation}
Since, $l_0<\frac{N(\kappa-q_0)}{N-2q_0}$ implies that
$$\frac{N}{2}\bigg(\frac{1}{l_0}-\frac{\kappa-l_0}{l_0q_0}\bigg)<1,$$
so that, by (i) of Lemma \ref{lemma-014}  with \dref{eq-012}, we have
\begin{equation}\begin{array}{rl}
\sup\limits_{t>0}\| v(\cdot,t)\|_{L^\frac{l_0q_0}{\kappa-l_0}(\Omega)}\leq  C_2(1+\sup\limits_{t>0}\|u(\cdot,t)\|_{L^{l_0}(\Omega)})\\
\end{array}
\label{eq-804}
\end{equation}
with some positive constant $C_2$.
Therefore, there is $C_3 > 0$ fulfilling
\begin{equation}\begin{array}{rl}
\sup\limits_{t>0}\| u(\cdot,t)\|_{L^{l_0}(\Omega)}\leq  C_3[1+(\sup\limits_{t>0}\|u(\cdot,t)\|_{L^{l_0}(\Omega)})^{\frac{q_0}{\kappa}}]\\
\end{array}
\label{eq-805}
\end{equation}
by using \dref{eq-803}.
Observing that $\frac{q_0}{\kappa}<1$ due to $\kappa>\frac{N}{2}>q_0$, we derive that there exists  a constant $\eta_1>0$ such that
\begin{equation}
\begin{array}{rl}
&\disp{\sup\limits_{t>0}\|u(\cdot,t)\|_{L^{{l_0}}(\Omega)}\leq\eta_1.}
\\
\end{array}
\label{eq-806}
\end{equation}
Now,   collecting  \dref{eq-012} and \dref{eq-806},
we derive that for some $r_0\geq 1$  satisfying $r_0> \frac{N}{2}$, there is a constant $\eta_2>0$ such that
\begin{equation}
\begin{array}{rl}
&\disp{\sup\limits_{t>0}\|u(\cdot,t)\|_{L^{{r_0}}(\Omega)}\leq\eta_2.}
\\
\end{array}
\label{eq-807}
\end{equation}
Notice that $\frac{N r_0}{N- r_0}>N$ due to $r_0> \frac{N}{2}$. Then resorting to the variation-of-constants formula for $v$ and $L^p$-$L^q$ estimates for the heat semigroup again, we derive that for $\theta\in(N,\frac{N r_0}{N- r_0})$, there exist a positive constant $C_4$ such that
\begin{equation}\begin{array}{rl}
&\|\nabla v(\cdot,t)\|_{L^\theta(\Omega)}\\
\leq&\disp \|\nabla e^{t(d\Delta-1)}v_0\|_{L^\theta(\Omega)} +\int_{0}^{t}\|\nabla e^{(t-s)(d\Delta-1)}u(\cdot,s)\|_{L^\theta(\Omega)} ds\\
\leq&\disp C_4\|\nabla v_0\|_{L^\infty(\Omega)} +C_4\int_{0}^{t}(1+(t-s)^{-\frac{1}{2}-\frac{N}{2}(\frac{1}{ r_0}-\frac{1}{\theta})})e^{-\lambda_1(t-s) }\|u(\cdot,s)\|_{L^ {r_0}(\Omega)} ds\\
\leq&\disp C_4\|\nabla v_0\|_{L^\infty(\Omega)} +C_4\eta_2 \int_{0}^{t}(1+(t-s)^{-\frac{1}{2}-\frac{N}{2}(\frac{1}{ {r_0}}-\frac{1}{\theta})})e^{-\lambda_1(t-s) } ds,\\
\end{array}
\label{eq-808}
\end{equation}
where $\lambda_1>0$ denotes the first nonzero eigenvalue of $-\Delta$ in $\Omega$ under Neumann boundary conditions. With the facts $v_0 \in W^{1,\infty}(\Omega)$ and
$$\int_{0}^{t}(1+(t-s)^{-\frac{1}{2}-\frac{N}{2}(\frac{1}{ {r_0}}-\frac{1}{\theta})})e^{-\lambda_1(t-s) } ds\leq \int_{0}^{\infty}(1+s^{-\frac{1}{2}-\frac{N}{2}(\frac{1}{ {r_0}}-\frac{1}{\theta})})e^{-\lambda_1s } ds<+\infty,$$
we find a constant $C_5 > 0$ such that
\begin{equation}
\int_{\Omega}|\nabla v|^{\theta} \leq C_5~~\mbox{for all}~~ t\in(0, T_{max})
\label{eq-809}
\end{equation}
by \dref{eq-808}, where $\theta\in(N,\frac{N {r_0}}{N- {r_0}})$.
Next, for any $p\geq\theta,$ taking  ${v^{p-1}}$ as a test function for the second  equation of \dref{eq-1} and using  the H\"{o}lder inequality  and \dref{eq-807} yields  that
\begin{equation}
\begin{array}{rl}
&\disp\frac{1}{p}\disp\frac{d}{dt}\|{v}\|^{{{p}}}_{L^{{p}}(\Omega)}+(p-1)
\int_{\Omega} {v^{p-2}}|\nabla v|^2+ \int_{\Omega} v^{p}\\
=&\disp{\int_{\Omega} uv^{p-1}}\\
\leq&\disp{\left(\int_{\Omega}u^{r_0}\right)^{\frac{1}{r_0}}\left(\int_{\Omega}v^{\frac{r_0}{r_0-1}(p-1)}\right)^{\frac{r_0-1}{r_0}}}\\
\leq&\disp{\eta_2\left(\int_{\Omega}v^{\frac{r_0}{r_0-1}(p-1)}\right)^{\frac{r_0-1}{r_0}}~~~\mbox{for all}~~t\in (0, T_{max}),}
\end{array}
\label{eq-810}
\end{equation}
where $\eta_2$ is the same as \dref{eq-807}.
Now, due to \dref{123}, in light of the Gagliardo--Nirenberg inequality and Young inequality, we derive that there exist positive constants $C_{6}$ and $C_{7}$  
such that
\begin{equation}
\begin{array}{rl}
\disp\eta_2\left(\int_{\Omega}v^{\frac{r_0}{r_0-1}(p-1)}\right)^{\frac{r_0-1}{r_0}} =&\disp{\eta_2\| { v^{\frac{p}{2}}}\|^{{\frac{2(p-1)}{p}}}_{L^{\frac{2r_0(p-1)}{p(r_0-1)}}(\Omega)}}\\
\leq&\disp{C_{6}\bigg(\| \nabla{ v^{\frac{p}{2}}}\|^{2\rho_1}_{L^{2}(\Omega)}\|{ v^{\frac{p}{2}}}\|^{{\frac{2(p-1)}{p}}-2\rho_1}_{L^{\frac{2}{p}}(\Omega)}+
\|{ v^{\frac{p}{2}}}\|_{L^{\frac{2}{p}}(\Omega)}^{{\frac{2(p-1)}{p}}} \bigg)}\\
\leq&\disp{C_{7}(\| \nabla{ v^{\frac{p}{2}}}\|^{2\rho_1}_{L^{2}(\Omega)}+1)}\\
\leq&\disp{\frac{4(p-1)}{p^2}\| \nabla{ v^{\frac{p}{2}}}\|^{2}_{L^{2}(\Omega)}+
C_{8}}\\
=&\disp{(p-1)
\int_{\Omega} {v^{p-2}}|\nabla v|^2+
C_{8}~~~\mbox{for all}~~t\in (0, T_{max}),}\\
\end{array}
\label{eq-811}
\end{equation}
where $$\rho_1=\frac{\frac{Np}{2}-\frac{Np(r_0-1)}{2r_0(p-1)}}{1-\frac{N}{2}+\frac{Np}{2}}$$
subject to the fact that  $\rho<1$ by $r_0>\frac{N}{2}$.
Inserting \dref{eq-811} into \dref{eq-810}, 
 we get a positive constant  $C_{9}$ such that
\begin{equation}
\begin{array}{rl}
\disp\frac{1}{p}\disp\frac{d}{dt}\int_{\Omega} v^{p}+\int_{\Omega} v^{p}
\leq & C_{9}~~~\mbox{for all}~~t\in (0, T_{max}),\\
\end{array}
\label{eq-812}
\end{equation}
which entails by the Gronwall inequality that
\begin{equation}
\begin{array}{rl}
\disp\int_{\Omega} v^{p}(\cdot,\cdot,t) \leq C_{10}~~\mbox{for all}~~ t\in(0, T_{max})~~~\mbox{and}~~~p\geq\theta.
\end{array}\label{eq-813}
\end{equation}
In view of $\theta>N$,  by the Sobolev embedding theorem  along with  \dref{eq-809} and \dref{eq-813}, we can find a constant
$C_{11} > 0$ such that
\begin{equation}
\begin{array}{rl}
&\disp{\|v(\cdot,t)\|_{L^{{\infty}}(\Omega)}\leq C_{11}~~~\mbox{for all}~~ t\in(0, T_{max}).}
\\

\end{array}
\label{eq-814}
\end{equation}
This along with the hypotheses $(H_1)$-$(H_2)$ gives us positive constants $C_{12}$ and  $C_{13}$ such that
\begin{equation}
\begin{array}{rl}
\disp C_{12}\leq\gamma(v),\phi(v)\leq &\disp{C_{13}~~~\mbox{for all}~~ t\in(0, T_{max}).}
\\
\end{array}
\label{eq-815}
\end{equation}
Next, we multiply the first equation of \dref{eq-1}
by $u^{p-1} (p > N+1)$ and integrate the resulting equation by parts to obtain some positive constants $C_{14}$ and $C_{15}$ so that
\begin{equation}\label{eq-816}
\begin{array}{rl}
\disp{\frac{1}{p}\frac{d}{dt}\| u  \|^{p}_{L^p(\Omega)}+C_{{14}}(p-1)\int_{\Omega} u  ^{p-2}|\nabla u  |^2{}{}}
\leq&\disp{C_{15}\int_{\Omega} u  ^{p}|\nabla v  |^2~~\mbox{for all}~~ t\in(0, T_{max})}\\
\end{array}
\end{equation}
by using the Young inequality and \dref{eq-815}. Next we estimate the term on the right hand side of \dref{eq-816}. In fact, with the H\"{o}lder inequality, one has
\begin{equation}
\begin{array}{rl}
\disp{ C_{15} \disp\int_\Omega u ^{p } |\nabla v  |^2}
&\leq\disp{ C_{15} \left(\disp\int_\Omega u ^{\frac{\varrho{p } }{\varrho-1}}\right)^{\frac{\varrho-1}{\varrho}}\left(\disp\int_\Omega |\nabla v  |^{2\varrho}\right)^{\frac{1}{\varrho}}}\\[3mm]
&\leq\disp{ C_{16}\|   u ^{\frac{p}{2}}\|^{2}_{L^{\frac{2\varrho}{\varrho-1}}(\Omega)}
,}\\
\end{array}
\label{eq-818}
\end{equation}
where $\varrho=\frac{\theta}{2}$ and $\theta$ is the same as \dref{eq-809}.
In view of  $p> N+1$, we have
$$\frac{1}{p}\leq{\frac{\varrho}{\varrho-1}}<\frac{N}{N-2}.$$
Furthermore an application of the Gagliardo-Nirenberg inequality along with Lemma \ref{lemma-010} yields  some positive constants $C_{17}$ and $ C_{18}$ such that
\begin{equation}
\begin{array}{rl}
\disp{C_{16}\|   u ^{\frac{{p}}{2}}\|
^{2}_{L^{\frac{2\varrho}{\varrho-1}}(\Omega)}}
\leq&\disp{C_{17}\Big(\|\nabla    u ^{\frac{{p}}{2}}\|_{L^2(\Omega)}^{2\rho_2}\| u ^{\frac{{p}}{2}}\|_{L^\frac{2}{{p}}(\Omega)}^{2-2\rho_2}+\| u ^{\frac{{p}}{2}}\|_{L^\frac{2}{{p}}(\Omega)}^{2}\Big)}\\[3mm]
\leq&\disp{C_{18}(\|\nabla    u ^{\frac{{p}}{2}}\|_{L^2(\Omega)}^{2\rho_2}+1),}\\
\end{array}
\label{eq-819}
\end{equation}
where $$\rho_2=\frac{Np-N+\frac{N}{\varrho}}{2-N+Np}.$$
Due to $\varrho>N/2$, one can check $0<\rho_2<1$. Noticing $\|\nabla    u ^{\frac{{p}}{2}}\|_{L^2(\Omega)}^{2}=\frac{4}{p^2}\int_\Omega u^{p-2}|\nabla u|^2$, then the Young inequality applied to $\|\nabla    u ^{\frac{{p}}{2}}\|_{L^2(\Omega)}^{2\rho_2}$ yields a constant $C_{19}>0$ so that we have from \eqref{eq-818}-\eqref{eq-819} that
$C_{15} \int_\Omega u ^{p } |\nabla v  |^2\leq \frac{C_{14}(p-1)}{2}\int_{\Omega} u  ^{p-2}|\nabla u  |^2+C_{19}$, which applied to \eqref{eq-816} leads to
\begin{equation}
\begin{array}{rl}
\disp{\frac{1}{p}\frac{d}{dt}\| u  \|^{p}_{L^p(\Omega)}+\frac{C_{14}(p-1)}{2}\int_{\Omega} u  ^{p-2}|\nabla u  |^2}
\leq&\disp{C_{19}~~\mbox{for all}~~ t\in(0, T_{max})}.
\end{array}
\label{eq-820}
\end{equation}
On the other hand, if we define  ${\rho_3=\frac{\frac{Np}{2}-\frac{N}{2}}{1-\frac{N}{2}+\frac{Np}{2}}}<1$ with $p>1$. Then by \dref{eq-012}, we can find positive constants  $C_{20}, C_{21}, C_{22}$  resulting from the Gagliardo-Nirenberg inequality and the Young inequality  such that
\begin{equation}
\begin{array}{rl}
\disp\int_{\Omega}u^{p}=&\disp{\|  u^{\frac{p}{2}}\|^{2}_{L^{2}(\Omega)}}\\
\leq&\disp{C_{20}(\|\nabla   u^{\frac{p}{2}}\|_{L^2(\Omega)}^{2\rho_3}\|  u^{\frac{p}{2}}\|_{L^\frac{2}{p }(\Omega)}^{2-2\rho_3}+\|  u^{\frac{p}{2}}\|_{L^\frac{2}{p }(\Omega)}^{2})}\\
\leq&\disp{C_{21}(\|\nabla   u^{\frac{p}{2}}\|_{L^{2}(\Omega)}^{2\rho_3}+1)}\\
\leq&\disp{\frac{C_{14}(p-1)}{2}\times\frac{4}{p^2}\|\nabla   u^{\frac{p}{2}}\|_{L^{2}(\Omega)}^{2}+C_{22}}\\
=&\disp{\frac{C_{14}(p-1)}{2}\int_{\Omega} u  ^{p-2}|\nabla u  |^2+C_{22}}.
\end{array}
\label{eq-821}
\end{equation}
Inserting \dref{eq-821} into \dref{eq-820}, we get a positive constant $C_{23}$ such that
\begin{equation}
\begin{array}{rl}
\disp\frac{1}{p}\disp\frac{d}{dt}\|{u}\|^{{{p}}}_{L^{{p}}(\Omega)}+\int_{\Omega}u^{p}
\leq & C_{23}~~\mbox{for all}~~ t\in(0, T_{max}),\\
\end{array}
\label{eq-822}
\end{equation}
which upon an use of Gronwall's inequality yields \eqref{LP} and hence completes the proof.
\end{proof}

\subsection{Proof of Theorem \ref{theorem3}}
To prove Theorem \ref{theorem3}, from the local existence theorem with extension criterion in Lemma \ref{lemma70}, it suffices to show that there is a constant $C>0$ independent of $t$ such that
\begin{equation}\label{eq-801}
\disp{\|u(\cdot,t)\|_{L^{\infty}(\Omega)}+\|\nabla v(\cdot,t)\|_{L^{\infty}(\Omega)}<C} \ \ \mathrm{for \ all} \ t \in (0, T_{max}).
\end{equation}
Indeed with Lemma \ref{lem8-01}, employing  the standard estimate for Neumann semigroup (see \cite[Lemma 1]{Kowalczyk}), one can find a positive constant $c_{1}$ independent of $t$ such that
\begin{equation*}
\begin{array}{rl}
\disp{\|v(\cdot,t)\|_{W^{1,\infty}(\Omega)}}
\leq\disp{c_{1}~~\mbox{for all}~~ t\in(0, T_{max}).}\\
\end{array}
\label{eq-824}
\end{equation*}
Since $\gamma(v)$ has a positive lower bound, see \eqref{eq-815}, one can employ the Moser iteration (see the proof of \cite[Lemma 3.5]{Wang-2}) to get a constant $c_2>0$ such that $\|u(\cdot,t)\|_{L^{\infty}(\Omega)}\leq c_2$. We hence get \eqref{eq-801} and complete the proof.

\subsection{Proof of Theorem \ref{theorem4}}
In the case of $\gamma(v)=\frac{\sigma}{v^\lambda} (\sigma, \lambda>0)$ with $\phi(v)=(\alpha-1)\gamma'(v)$ for $0<\alpha<1$, we can compute that
$$F(v):=\frac{d\gamma(v)}{v\phi(v)(v\phi(v)+d-\gamma(v))_+}=\frac{d}{\lambda(1-\alpha)\Big(\frac{[\lambda(1-\alpha)-1]\sigma}{v^\lambda}+d\Big)_+}$$
where the hypothesis $(H_3)$ is equivalent to
\begin{equation}\label{H3E}
\inf_{v\geq0} F(v)>\frac{N}{2}.
\end{equation}
Then we have two cases to proceed.

{\bf Case 1}. If $\lambda(1-\alpha)-1>0$, namely $\lambda>\frac{1}{1-\alpha}$, we have from Lemma \ref{lemm4-02} that
$$\inf_{v\geq0} F(v)>\frac{N}{2} \Leftrightarrow \frac{d}{\lambda(1-\alpha)\Big(\frac{[\lambda(1-\alpha)-1]\sigma}{\eta^\lambda}+d\Big)}>\frac{N}{2}.$$

{\bf Case 2}. If $\lambda(1-\alpha)-1\leq 0$, namely $\lambda\leq\frac{1}{1-\alpha}$, then we can check
$$\sup\limits_{v\geq 0}\bigg(\frac{[\lambda(1-\alpha)-1]\sigma}{v^\lambda}+d\bigg)_+=d$$
and hence
$$\inf_{v\geq0} F(v)=\frac{1}{\lambda(1-\alpha)}.$$
Then the condition \eqref{H3E} becomes
$$\frac{1}{\lambda(1-\alpha)}>\frac{N}{2}.$$

Combining the results in Case 1 and Case 2 along with Theorem \ref{theorem3} completes the proof of Theorem \ref{theorem4}.


\bigbreak

\noindent {\bf Acknowledgement}:
The research of Z.A. Wang was supported by the Hong Kong RGC GRF grant No. 15303019 (Project ID P0030816) and an internal grant no. UAH0 from the Hong Kong Polytechnic University. The work of J. Zheng was partially supported by  Shandong Provincial
Science Foundation for Outstanding Youth (No. ZR2018JL005), the National Natural
Science Foundation of China (No. 11601215)  and Project funded by China
Postdoctoral Science Foundation (No. 2019M650927, 2019T120168).


\begin{thebibliography}{plain}
\normalsize
\bibitem{Ahn-1} J. Ahn, C. Yoon, \textit{Global well-posedness and stability of constant equilibria in parabolic-
elliptic chemotaxis systems without gradinet sensing},  Nonlinearity, 32(2019), 1327--1351.


\bibitem{Amann-book} H. Amann, Nonhomogeneous linear and quasilinear elliptic and parabolic boundary value problems. In Function spaces, differential operators and nonlinear analysis (Friedrichroda, 1992), volume 133 of Teubner-Texte Math., pages 9-126. Teubner, Stuttgart, 1993.




\bibitem{Bellomo-1} N. Bellomo,  A. Belloquid,   Y. Tao, M. Winkler,  \textit{Toward a mathematical theory of
Keller--Segel models of pattern formation in biological tissues}, Math. Models Methods Appl. Sci., 25(2015), 1663--1763.

\bibitem{Bilerss-1} P. Biler, \textit{Global solutions to some parabolic-elliptic systems of chemotaxis}, Adv. Math. Sci. Appl., 9(1999), 347--359.

%
%
%
%



\bibitem{Burger}M. Burger, P. Lauren\c{c}ot, A. Trescases, Delayed Blow-Up for Chemotaxis Models with Local Sensing, arXiv:2005.02734v2, 2020.


\bibitem{Kim-NARWA-2019}L. Desvillettes, Y.J. Kim, A. Trescases, C. Yoon, \textit{ A logarithmic chemotaxis model featuring global existence and aggregation}. Nonlinear Anal. RWA., 50(2019), 562--582.


\bibitem{Fujie-5}K. Fujie, \textit{Boundedness in a fully parabolic chemotaxis system with singular sensitivity},  J. Math. Anal. Appl., 424(2015),
675--684.

\bibitem{FujieJiang2020-1}K. Fujie, J. Jiang, \textit{Comparison methods for a Keller-Segel-type model of pattern formations with density-suppressed motilities}, arXiv:2001.01288, 2020.

\bibitem{FujieJiang2020-2} K. Fujie, J. Jiang, \textit{Global existence for a kinetic model of pattern formation with density-suppressed motilities}. J. Diff. Eqns., 269(2020), 5338--5378.



\bibitem{FuFuFugttt}  X. Fu, L. Tang, C. Liu, J.D. Huang, T. Hwa, P. Lenz, \textit{Stripe formation in bacterial
system with density-suppressed motility}, Phys. Rev. Lett., 108(2012), 198102.


\bibitem{Fujie-1} K. Fujie, T. Senba,  \textit{A sufficient condition of sensitivity functions for boundedness of solutions to a parabolic-parabolic chemotaxis system}, Nonlinearity, 31(2018), 1639--1672.

\bibitem{Fujie-2} K. Fujie, T. Senba,  \textit{Global existence and boundedness of radial solutions to a two dimensional fully parabolic chemotaxis system with general sensitivity}, Nonlinearity, 28(2016), 2417-2450.












%
\bibitem{Horstmann2710}D. Horstmann,  \textit{From $1970$ until present: the Keller--Segel model in chemotaxis and its consequences,} I.
Jahresberichte der Deutschen Mathematiker-Vereinigung, 105(2003), 103--165.
%
%
%
%


\bibitem{Kowalczyk} R. Kowalczyk, Z. Szyma\'{n}ska, \textit{On the global existence of solutions to an aggregation
model}. J. Math. Anal. Appl., 343(2008), 379--398.

\bibitem{Jin-1}  H. Jin, Y. Kim, Z.A. Wang, \textit{Boundedness, stabilization, and pattern formation
driven by density-suppressed motility}, SIAM J. Appl. Math., 78(2018), 1632--1657.

\bibitem{Jin-2} H. Jin, Z.A. Wang, \textit{Global dynamics and spatio-temporal patterns of predator-prey
systems with density-dependent motion}, Euro. J. Appl. Math., doi:10.1017/S0956792520000248, 2020.

\bibitem{Jin-3} H. Jin, Z.A. Wang, \textit{Critical mass on the Keller-Segel system with signal-dependent
motility}, Proc. Amer. Math. Soc., DOI: https://doi.org/10.1090/proc/15124, 2020.

\bibitem{JW-DCDSB-2020} H. Jin, Z.A. Wang,  \textit{The Keller-Segel system with logistic growth and signal-dependent motility},   Disc. Cont. Dyn. Syst.-B, doi: 10.3934/dcdsb.2020218, 2020.




\bibitem{Keller-1}E.F. Keller, L.A. Segel,  \textit{Model for chemotaxis}, J. Theor. Biol., 30(1970),  225--234.

 \bibitem{Keller-2} E.F Keller, L.A. Segel, \textit{Initiation of slime mold aggregation viewed as an instability, }  J. Theor. Biol., 26(1970), 399--415.

 \bibitem{Keller-3} E.F. Keller, L.A. Segel, \textit{Traveling bands of chemotactic bacteria:  A theorectical analysis}, J. Theor. Biol., 26(1971), 235-248.


\bibitem{Lankeit-1} J. Lankeit,  \textit{A new approach toward boundedness in a two-dimensional parabolic chemotaxis system with
singular sensitivity},  Math. Methods Appl. Sci., 39(2016), 394--404.

 \bibitem{Lankeit-11} E. Lankeit, J. Lankeit,  \textit{Classical solutions to a logistic chemotaxis model with singular sensitivity and signal absorption}, Nonlinear Anal. RWA., 46(2019),  421--445.


\bibitem{Li-Wang-2020} J. Li and Z.A. Wang, {\it Traveling waves on density-suppressed motility models}, arXiv:2006.12851, 2020.

\bibitem{Liu-1}  C. Liu, X. Fu, L. Liu, X. Ren, C.K.L. Chau, S. Li, H. Zeng, G. Chen, L. Tang, P. Lenz, X. Cui, W. Huang, T. Hwa, and J. Huang, \textit{ Sequential establishment of stripe patterns in an expanding cell population}.  Science, 334(2011):238--241.

\bibitem{Lui} R. Lui, H. Ninomiya, \textit{Traveling wave solutions for a bacteria system with densisuppressed motility}, Disc. Cont. Dyn. Syst.-B, 24(2018): 931-940.

\bibitem{Ma-1}  M. Ma, R. Peng, Z.A. Wang, \textit{Stationary and non-stationary patterns of the density-
suppressed motility model}, Phys. D, 402(2020), 132259, 13 pages.

\bibitem{Nagai-4}T. Nagai,  T. Senba,  \textit{Behavior of radially symmetric solutions of a system related to chemotaxis},  Nonlinear Anal., 30(1997),
3837--3842.

\bibitem{Nagai-3} T. Nagai, T. Senba, \textit{Global existence and blow-up of radial solutions to a parabolic-elliptic system of
chemotaxis},  Adv. Math. Sci. Appl., 8(1998), 145--156.




\bibitem{Nanjundiah} V. Nanjundiah, \textit{Chemotaxis, signal relaying and aggregation morpholog,} J. Theor. Biol., 42(1973), 63--105.

\bibitem{Nirenber-1} L. Nirenberg, \textit{On elliptic partial differential equations}, Ann. Sc. Norm. Super. Pisa, Sci. Fis. Mat.,
III. Ser., 13(1959), 115--162.



%
%
%

 \bibitem{Smith-1} J. Smith-Roberge, D. Iron, T. Kolokolnikov, \textit{Pattern formation in bacterial colonies
with density-dependent diffusion}, Eur. J. Appl. Math., 30(2019), 196--218.

\bibitem{Stinn-1} C. Stinner, M. Winkler, \textit{Global weak solutions in a chemotaxis system with large singular sensitivity},
Nonlinear Anal. RWA., 12(2011), 3727--3740.


%
%



\bibitem{Wang-1} J. Wang, M. Wang, \textit{Boundedness in the higher-dimensional Keller-Segel model with
signal-dependent motility and logistic growth}, J. Math. Phys., 60(2019), 011507.


%
\bibitem{Wang-review} Z.A. Wang, \textit{Mathematics of traveling waves in chemotaxis},
Discrete Contin. Dyn. Syst- B.,18(2013): 601-641.

\bibitem{Wang-2} Z.A. Wang, \textit{On the parabolic-elliptic Keller-Segel system with signal-dependent motilities: a paradigm for global boundedness and steady states}, arXiv:2005.04415, 2020.


%

 \bibitem{Winkler-1} M. Winkler, \textit{Aggregation vs. global diffusive behavior in the higher-dimensional Keller--Segel model}, J. Diff.
Eqns., 248(2010), 2889--2905.


\bibitem{Winkler-5} M. Winkler,  \textit{Global solutions in a fully parabolic chemotaxis system with singular sensitivity}, Math. Methods
Appl. Sci., 34(2011), 176--190.



%
%


\bibitem{Yoon-1} C. Yoon, Y.J. Kim, \textit{Global existence and aggregation in a Keller-Segel model with
Fokker-Planck diffusion}, Acta Appl. Math., 149(2017), 101--123, .



\bibitem{Zheng0} J. Zheng, \textit{Boundedness of solutions to a quasilinear parabolic--elliptic Keller--Segel system with logistic source},  J. Diff. Eqns.,
259(2015),  120--140.



\end{thebibliography}
\end{document}